\documentclass[times]{my-iapress}

\usepackage{moreverb}
\usepackage[colorlinks,bookmarksopen,bookmarksnumbered,citecolor=red,urlcolor=red]{hyperref}
\usepackage{tikz}
\usepackage{subfigure}

\def\volumeyear{2019}
\def\volumenumber{7}
\def\volumemonth{June}
\usepackage{amsmath,amssymb}
\usepackage{graphicx}

\newtheorem{theorem}{Theorem}
\newtheorem{lemma}{Lemma}

\newcommand{\R}{\mathbb{R}}

\allowdisplaybreaks

\begin{document}

\runningheads{Analysis of a SIRI epidemic model with distributed delay and relapse}{
A. Elazzouzi, A. Lamrani Alaoui, M. Tilioua and D. F. M. Torres}

\title{Analysis of a SIRI epidemic model with distributed delay and relapse}

\author{Abdelhai Elazzouzi\affil{1}\corrauth,
Abdesslem Lamrani Alaoui\affil{2},
Mouhcine Tilioua\affil{2},
Delfim F. M. Torres\affil{3}}

\address{\affilnum{1}Department of MPI, University Sidi Mouhamed Ben Abdellah,
FP Taza, LSI Laboratory, Morocco.
\affilnum{2}Department of Mathematics, University Moulay Isma\"{\i}l,
FST Errachidia, M2I Laboratory, MAMCS Group, Morocco.
\affilnum{3}Center for Research and Development in Mathematics and Applications (CIDMA),
Department of Mathematics, University of Aveiro, 3810-193 Aveiro, Portugal.}

\corraddr{A. Elazzouzi (Email: abdelhai.elazzouzi@usmba.ac.ma),
Sidi Mouhamed Ben Abdellah, FP Taza, LSI Laboratory, Morocco.}

\begin{abstract}
We investigate the global behaviour of a SIRI epidemic model
with distributed delay and relapse. From the theory
of functional differential equations with delay,
we prove that the solution of the system is unique,
bounded, and positive, for all time. The basic reproduction
number $R_{0}$ for the model is computed. By means of the direct
Lyapunov method and LaSalle invariance principle,
we prove that the disease free equilibrium is globally
asymptotically stable when $R_{0} < 1$. Moreover,
we show that there is a unique endemic equilibrium,
which is globally asymptotically stable,
when $R_{0} > 1$.
\end{abstract}


\maketitle



\section{Introduction}

In recent years, great attention has been paid to the study of SIR type models,
which have been formulated to describe the propagation and evolution
of some human or animal diseases. In such models, the population
is subdivided into compartments or classes, in particular the compartment
of susceptible $(S)$, the compartment of infective $(I)$,
and the compartment of recovered individuals $(R)$
\cite{MR3703345,MR3810808}. When recovered individuals may experience
a relapse of the disease, due to an incomplete treatment or due to the reactivation
of a latent infection, and then re-enter the class of infective,
a SIRI model is more convenient to model the dynamic of the diseases.
Herpes, which can be transmitted by close physical or sexual contact,
tuberculosis, and malaria, are three epidemics that have been modeled
by SIRI systems \cite{MR3750672,Sampath,article,Xu,ZhangHong}.
Often, the transmission of the infection in the population is modelled
by an incidence function, which has taken many forms in the literature
\cite{enatsu2012global,Georgescu2013ALF,korobeinikov2004lyapunov}.
Most epidemiological models focus on an incidence function without delay
\cite{MyID:382}. They assume that infection occur instantaneously once there is
a contact between an infectious individual and a susceptible one.
On the other hand, some models incorporate an incidence function
with discrete or distributed delays to model latency
\cite{MR2921926,MR1218880,LI20123696}.

In the analysis of local or global asymptotic stability properties,
it is common to use Lyapunov's second method, also called the
direct method of Lyapunov. It is a robust tool that allows
to determine the stability of a system without explicitly integrating
the differential equation. For details, see, e.g.,
\cite{beretta2001global,korobeinikov2006lyapunov,%
lasalle1976stability,tagkey1961iii,yasuhiro2000global}.
Several works include relapse. In \cite{Tudor},
a constant population SIRI model
with the incidence function $S(t)I(t)$ is analysed.
In \cite{Moreira}, an extension of the model in \cite{Tudor} for herpes viral infections
is investigated. It is proved that both disease-free and endemic equilibria
of the model are globally asymptotically stable \cite{Moreira}. In \cite{Georgescu2013ALF},
a SIRI epidemic model with the incidence rate of infection $C(S)f(I)$ is studied.
Sufficient conditions for the local stability of the equilibria are given
by using Lyapunov's second method and, under suitable monotonicity conditions,
global stability is obtained. In \cite{vargas2014global},
the global stability of a SIRI model with constant recruitment,
disease-induced death, and bilinear incidence rate, is discussed.
Here, we consider a SIRI model with relapse, a distributed delay,
and a general non-linear incidence function. The direct method of Lyapunov
is used to prove global asymptotic stability for any steady state.

The paper is organized as follows. The mathematical model under consideration
is formulated in Section~\ref{sec:02}. In Section~\ref{sec:03},
we establish its well-posedness.
More precisely, we prove positivity and boundedness of the solution
(Theorem~\ref{thm:pos:bound}).
The basic reproduction number and the disease-free equilibrium $E_0$
are also determined (Theorem~\ref{thm:R0}). In Section~\ref{sec:04},
we provide a mathematical analysis of the model.
In particular, the global stability of the
disease-free equilibrium and the global stability of the endemic
equilibrium are shown (Theorems~\ref{theorem1} and \ref{thm:Stab:EE}).
Two numerical examples with an incidence
function satisfying the assumptions considered in the previous sections
are given in Section~\ref{sec:05}. We finish the paper with
Section~\ref{sec:06} of concluding remarks
and some perspectives for future research.


\section{The mathematical model}
\label{sec:02}

We consider a general SIRI epidemic model with distributed delay and relapse.
The flow diagram of the disease transmission is given in Figure~\ref{figure0},
\begin{figure}[ht!]
\centering{
\tikzstyle{transition}=[rectangle, draw=blue!50,fill=red!20,thick,inner sep=0pt,minimum size=8mm]
\begin{tikzpicture}
\node [ draw , shape = rectangle ] (A) at (1.5, 0) [transition] {$s$};
\node [draw , shape = rectangle ] (B) at (7, 0) [transition] {$i$};
\node [draw , shape = rectangle ] (C) at (3.8, -2) [transition] {$r$};
\draw (A) edge [ dashed , ->] (B);
\draw (B) edge [ dashed , ->] (C);
\draw (B) edge [ dashed , ->] (7,1.5);
\draw (B) edge [ dashed , ->] (8.5,0);
\draw[->] (0,0) -- (A);
\draw[->] (A) -- (1.5,-1.5);
\draw[->] (C) -- (3.8,-3.5);
\draw[->] (C) -- (7,-2)--(B);
\path (A) edge node[anchor=south,above]{$\beta \displaystyle \int_0^{h}g(\tau)f(s(t),i(t-\tau))d\tau$}
(B);
\path (0,0) edge node[anchor=south,above]{$\Lambda$}
(A);
\path (B) edge node[anchor=south,left]{$\mu i$}
(7,1.5);
\path (B) edge  node[anchor=south,above]{$c i$}
(8.5,0);
\path (B) edge  node[anchor=south,left]{$\gamma i$}
(C);
\path (7,-2) edge   node[anchor=sought,right]{$ \delta r$}
(B);
\path (A) edge   node[anchor=sought,right]{$ \mu s$}
(1.5,-1.5);
\path (C) 	edge   node[anchor=sought,right]{$ \mu r$}
(3.8,-3.5);
\end{tikzpicture}
}
\caption{Flow diagram of the disease transmission.}	
\label{figure0}
\end{figure}
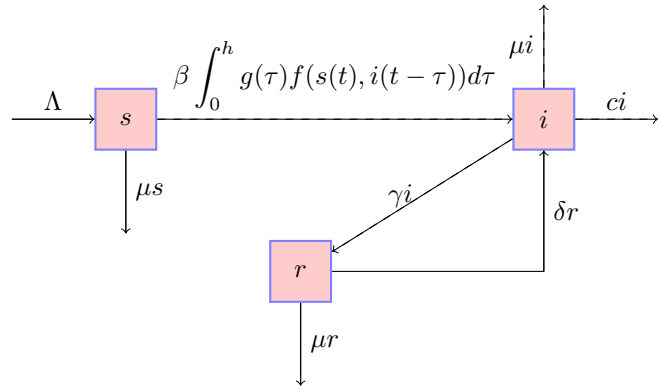
which corresponds to the dynamics described
by the system of equations \eqref{Model}:
\begin{eqnarray}
\label{Model}
\left\{
\begin{array}{ll}
\dfrac{ds(t)}{dt}
=\Lambda- \mu s(t)- \beta
\displaystyle \int_0^{h}g(\tau)f(s(t),i(t-\tau))d\tau,\vspace{.1cm}\\
\dfrac{di(t)}{dt}= \beta \displaystyle \int_0^{h}g(\tau)
f\left(s(t),i(t-\tau)\right)d\tau
-(\mu+c+\gamma)i(t) + \delta r(t), \vspace{.1cm}\\
\dfrac{dr(t)}{dt}=\gamma i(t)- (\mu + \delta) r(t),
\end{array}
\right.
\end{eqnarray}
where $s(t)$, $i(t)$, and $r(t)$ denote, respectively,
the number of susceptible, infective, and recovered 
individuals at time $t$. The parameters of model
\eqref{Model} are summarized in Table~\ref{table:01}.
\begin{table}[!ht]
\caption{Description of the parameters of the SIRI model \eqref{Model}.}
\label{table:01}
\centering{
\begin{tabular}{c|l} \hline
Parameter & Biological meaning\\ \hline
$\Lambda$& the population recruitment rate\\
$\mu$ & the population natural death rate\\
$\gamma$ & the natural recovery rate of infective individuals\\
$c$ & the population death rate caused by infection\\
$\beta$ &  the transmission coefficient\\
$\delta$ & the relapse rate\\ \hline
\end{tabular}}
\end{table}
Individuals leave the susceptible class at a rate
$$
\int_0^{h}g(\tau) f\left(s(t),i(t-\tau)\right) d\tau,
$$
where $h$ represents the maximum time taken to become infectious 
and $g$ denotes the fraction of vector population in which the time 
taken to become infectious is $\tau$ (that is, the incubation period distribution), 
which is assumed to be a non-negative continuous function on $[0, h]$. 
Moreover, and without loss of generality, we assume that
$$
\int_0^{h}g(\tau)d\tau = 1.
$$
Otherwise, we consider 
$$
\underline{g}=\dfrac{g}{\displaystyle\int_0^{h}g(\tau)d\tau} 
\quad \text{and} \quad 
\underline{\beta}=\beta\displaystyle\int_0^{h}g(\tau)d\tau
$$ 
instead of $g$ and $\beta$, respectively.
The initial conditions for system \eqref{Model} 
are given for $\theta \in[-h,0] $ by
$$
s( \theta ) = \Phi_1 ( \theta ), \quad
i( \theta ) = \Phi_2 ( \theta ), \quad
r( \theta ) = \Phi_3 ( \theta )
$$
with $\Phi = \left(\Phi_1, \Phi_2, \Phi_3\right) \in~ C^{+} \times~ C^{+}\times~ C^{+}$,
where $C^{+} = C([-h,0],\mathbb{R}^{+})$  is the non-negative cone of $C$.
Here, $C:= C([-h,0],\mathbb{R})$ denotes the space of continuous functions mapping
$[-h,0]$ into $\mathbb{R}$, equipped with the sup-norm.
Let $s(0) = s_0 > 0$, $i(0) = i_0>0$, and $r(0) = r_0 >0$.
Our main objective is to investigate the global stability of the SIRI model \eqref{Model}.
For that, we construct suitable Lyapunov functionals.


\section{The well-posedness of the model and its basic reproduction number}
\label{sec:03}

Let $f:\R^2_+ \rightarrow \R_+$ be a continuously differentiable function
in the interior of $\R^2_+$, satisfying
the following hypotheses:
\begin{description}
\item[$(H_1)$] $ f(s,i) $ is a strictly monotone increasing
function of $s \geqslant 0$ for any fixed $ i > 0$
and a monotone increasing function of $ i > 0 $
for any fixed $ s \geqslant 0$;

\item[$(H_2)$] $\phi(s, i) = \dfrac{f(s,i )}{i}$
is a bounded and monotone decreasing function of $i > 0$
for any fixed $s \geqslant 0$;

\item[$(H_3)$] $f(0, i) = f(s, 0) = 0$ for all $s, i\geqslant 0$.
\end{description}
Then, the right side of \eqref{Model} is locally Lipschizian
and we get, from the classical theory of ODEs with delay,
local existence and uniqueness of solution for \eqref{Model},
i.e., existence and uniqueness for all $ t \in [ 0, \delta ]$
for some $\delta \geq 0$ \cite{hale,hale2}. It is also easy
to see that system \eqref{Model} has always a disease-free equilibrium
\begin{equation}
\label{eq:DFE}
E_{0} = \left(\dfrac{\Lambda}{\mu}, 0, 0\right).
\end{equation}
We begin by proving that our model \eqref{Model}
is not only mathematically but also biologically well-posed:
all feasible solutions of system \eqref{Model}
are bounded and positive.

\begin{theorem}
\label{thm:pos:bound}
Let $\left(s(t),i(t),r(t)\right)$ be any solution of model
\eqref{Model} with positive initial condition $(s_0,i_0,r_0)$. Then,
\begin{enumerate}
\item every solution of \eqref{Model} starting from
$(s_0,i_0,r_0)$ remains positive for all $t\geq 0$,

\item the set
$$
\Omega = \left\{(s,i,r)\in \R^3 : \, s>0, \, i>0,
\, r>0, \, s+i+r\leq \dfrac{\Lambda}{\mu}\right\}
$$
is positively invariant with respect to system \eqref{Model}.
\end{enumerate}
\end{theorem}

\begin{proof}
Assume, by contradiction, that the first item of our result is false.
Let $t_1 = \min\{t : s(t)i(t) =0 \}$. Assume that $ s(t_1) =0$,
which implies that $i(t) \geq 0$ for all $0 \leq t \leq t_1 $. Let
\begin{eqnarray*}
A= \min_{ 0 \leq t \leq t_1}  \left\{ \dfrac{\Lambda}{s(t)}
- \mu - \beta \displaystyle\int_0^{h}
g(\tau) \dfrac{f(s(t),i(t-\tau))}{s(t)}d\tau \right\}.
\end{eqnarray*}
It follows that $\dfrac{d}{dt}s(t) \geq A s(t)$. Therefore,
$s(t_1) \geq s(0) \exp(At_1) > 0$. This contradicts
$ s(t_1) = 0$. With a similar argument, we see that $ i(t_1) = 0$
is a contradiction. This proves that $ s(t) > 0 $ and $ i(t) > 0 $
for all $ t \geq 0 $. On the other hand, from the third equation of
\eqref{Model}, one has $\frac{d}{dt}r(t) \geq - (\mu+\delta) r(t)$,
which implies $$r(t) \geq r(0) \exp (- (\mu+\delta) t) > 0.$$
We have just proved that any solution $(s(t),i(t),r(t))$ is positive.
Now, let $N(t)=s(t)+i(t) +r(t)$. Then,
\begin{eqnarray*}
\displaystyle\dfrac{d}{dt}N(t)= \Lambda- \mu N(t)-ci(t).
\end{eqnarray*}
It follows that
\begin{eqnarray*}
\displaystyle\dfrac{d}{dt}N(t)\leq \Lambda-\mu N(t).
\end{eqnarray*}
Now, let $V$ be the unique solution of the initial value problem
\begin{equation*}
\left\{
\begin{array}
[c]{l}
\displaystyle\dfrac{d}{dt}V(t)= \Lambda - \mu V(t)\text{ \ \ for } t>0,\\[0.3cm]
V(0)=N(0).
\end{array}
\right.
\end{equation*}
Then, $V$ is given by
$$
V(t)= \dfrac{\Lambda}{\mu} (1- \exp(-\mu t)) + N(0) \exp(-\mu t)
$$
and, by the comparison theorem (see Theorem~\ref{thm:05:App} in Appendix), it follows that
\begin{eqnarray*}
N(t)\leq \dfrac{\Lambda}{\mu}\left(1-\exp(-\mu t)\right)+N(0) \exp\left(-\mu t\right).
\end{eqnarray*}
This implies that the solution is bounded and, by the blow-up phenomena,
the solution exists and is defined for all $t\geq0$. Moreover,
for $t$ going to $+\infty$, we have
$$
0 \leq N(t)\leq \dfrac{\Lambda}{ \mu}.
$$
Since the solution is positive and bounded, we conclude
that $\Omega$ is positively invariant with respect to \eqref{Model}.
\end{proof}

\begin{theorem}
\label{thm:R0}
The basic reproduction number of model \eqref{Model} is given by
\begin{eqnarray*}
R_{0}=\dfrac{\beta(\mu+\delta) \, \partial_2 f(E_0)}{(\mu+\delta)(\mu+c+\gamma)-\delta\gamma},
\end{eqnarray*}
where $\partial_2 f(E_0)$ denotes the partial derivative of $f$ with respect to its second
argument $i$ at the disease-free equilibrium $E_0$ given by \eqref{eq:DFE}.
\end{theorem}

\begin{proof}
We obtain the basic reproduction number by means of
the next generation method as given in \cite{VANDENDRIESSCHE200229}.
Let $x=(i,r,s)$. Then, it follows from system \eqref{Model} that
\begin{eqnarray*}
\dfrac{dx}{dt}= \mathcal{F}- \mathcal{ \nu},
\end{eqnarray*}
where
$$
\mathcal{F}=
\left(
\begin{matrix}
\beta \displaystyle \int_0^{h}g(\tau)f(s(t),i(t-\tau))d\tau\\[2ex]
0\\[2ex]
0
\end{matrix}
\right)
$$
and
$$
\mathcal{\nu}=
\left(
\begin{matrix}
(\mu +c+\gamma)i(t)- \delta r(t)\\[2ex]
-\gamma i(t)+ (\mu+\delta)r(t)\\[2ex]
\beta \displaystyle \int_0^{h}g(\tau)f(s(t),i(t-\tau))d\tau - \Lambda +\mu s(t)
\end{matrix}
\right).
$$
At the disease-free equilibrium $E_0$,
\begin{equation*}
D\mathcal{F}(E_0)
= \left(
\begin{matrix}
F&O_{2,1}\\[2ex]
O_{1,2}&0
\end{matrix}
\right),\quad
D\mathcal{\nu}(E_0)
=\left(
\begin{matrix}
V&O_{2,1}\\[2ex]
J_{1}&J_{2}
\end{matrix}
\right),
\end{equation*}
where the infection matrix $F$ and the transition matrix $V$ are given by
$$
F=
\left(
\begin{matrix}
\beta \partial_2 f(E_0)&0\\[2ex]
0&0
\end{matrix}
\right)
\quad
\text{ and }
\quad
V= \left(
\begin{matrix}
\mu +c+\gamma &\delta \\[2ex]
-\gamma &\mu +\delta
\end{matrix}
\right).
$$
The inverse of $V$ is
\begin{center}
$V^{-1}= \displaystyle \frac{1}{ (\mu+\delta)(\mu+c+\gamma)-\delta\gamma}$
$
\left(
\begin{matrix}	
\mu + \delta & \delta \\[2ex]
\gamma & \mu+ c +\delta
\end{matrix}
\right).
$
\end{center}
Thus, the next generation matrix for system \eqref{Model} is
$$
FV^{-1} =\displaystyle \frac{\beta \partial_2 f(E_0)}{ 
(\mu+\delta)(\mu+c+\gamma)-\delta\gamma}
\left(
\begin{matrix}
\mu + \delta & - \delta \\[2ex]
0& 0
\end{matrix}
\right).
$$
The basic reproduction number $R_{0}$ is the spectral radius
of the matrix $ FV^{-1}$, and the result follows.
\end{proof}


\section{Analysis of the model}
\label{sec:04}

In this section, we prove that there exists
a unique endemic equilibrium when the basic
reproduction number given by Theorem~\ref{thm:R0}
is greater than one (Lemma~\ref{Lemma:EE}),
and we obtain conditions for which the disease-free
equilibrium and the endemic equilibrium
are globally asymptotically stable
(Theorems~\ref{theorem1} and \ref{thm:Stab:EE}, respectively).


\subsection{Existence of an endemic equilibrium }

In this section, we establish existence and uniqueness of an endemic equilibrium.

\begin{lemma}
\label{Lemma:EE}
If  $ R_0 > 1 $, then system \eqref{Model} admits a unique
endemic equilibrium $ E^* = (s^*, i^*,r^*)$.
\end{lemma}

\begin{proof}
We look for the solutions $\left(s^*, i^*,r^*\right)$ of equations
$\dfrac{ds}{dt} = 0$, $\dfrac{di}{dt} = 0$, and $\dfrac{dr}{dt} = 0$.
First note that $ \dfrac{ds}{dt} + \dfrac{di}{dt} =0 $ implies
\begin{eqnarray*}
\Lambda - \mu s^* -(\mu+c+\gamma)i^*+ \dfrac{\delta\gamma}{\mu+\delta}i^* = 0
\end{eqnarray*}
and so
\begin{eqnarray*}
s^*=~\dfrac{\Lambda}{\mu}-\dfrac{(\mu+c+\gamma)-\dfrac{\delta\gamma}{\mu+\delta}}{\mu}i^*.
 \end{eqnarray*}
Let $H$ be the function defined from $\mathbb{R}^+$ to $\mathbb{R}$ by
\begin{eqnarray*}
H(i)=\beta \dfrac{f\left(\dfrac{\Lambda}{\mu}-\dfrac{(\mu+c+\gamma)
-\dfrac{\delta\gamma}{\mu+\delta}}{\mu}i,i \right)}{i}
-\left(\mu+c+\gamma\right)+ \dfrac{\delta\gamma}{\mu+\delta}.
\end{eqnarray*}
It follows that $H$ satisfies
\begin{eqnarray*}
\lim\limits_{i \rightarrow 0^+} H(i)
= \beta \partial_2 f(E_0)-(\mu+~c+\gamma)+\dfrac{\delta\gamma}{\mu+\delta}
= \left((\mu+c+\gamma)-\dfrac{\delta\gamma}{\mu+\delta}\right)(R_{0}-1)> 0
\end{eqnarray*}
and
\begin{eqnarray*}
H \left(\dfrac{\Lambda(\mu+\delta)}{(\mu+\delta)(\mu+c+\gamma)-\delta\gamma}\right)
=\dfrac{\delta\gamma}{\mu+\delta} -(\mu+c+\gamma) <0.
\end{eqnarray*}
Then, by the intermediate value theorem, there exists at least a $i^*$ such that 
$$ 
0 < i^* <\dfrac{\Lambda(\mu+\delta)}{(\mu+\delta)(\mu+c+\gamma)-\delta\gamma} 
\quad \mbox{\ and\ } \quad H(i^*)=0.
$$
Moreover, hypothesis $(H_2)$ implies that $H$ is a strictly monotone decreasing 
function on $\mathbb{R}^+$. Then, we conclude with the existence and uniqueness 
of $i^*$ such that
$$ 
0 < i^* <\dfrac{\Lambda(\mu+\delta)}{(\mu+\delta)(\mu+c+\gamma)-\delta\gamma} 
\quad \mbox{\ and\ }\quad H(i^*)=0.
$$
Furthermore, since $ \displaystyle \int_0^{h}g(\tau)d\tau = 1$, it follows that
$$
\beta \displaystyle \int_0^{h}g(\tau)f(s^*,i^*)d\tau=\Lambda-\mu s^*
$$ 
and 
$$
\beta \displaystyle \int_0^{h}g(\tau)
f\left(s^*,i^*\right)d\tau=(\mu+c+\gamma)i^*-\delta r^*,
$$ 
where 
$$ 
r^*=\displaystyle\dfrac{\gamma}{\mu + \delta}i^*.
$$
Therefore, $E^* = (s^*, i^*,r^*)$ is the unique endemic equilibrium of system \eqref{Model}. 
\end{proof}


\subsection{Global stability of the disease-free equilibrium}

We define a Lyapunov functional,
showing the global asymptotic stability
of the disease-free equilibrium $E_0$
of system \eqref{Model}.

\begin{theorem}
\label{theorem1}
Assume that the hypotheses $(H_1)$ and $(H_2)$ hold. Then, the disease
free equilibrium $E_{0}$ of system \eqref{Model} is globally asymptotically 
stable if, and only if, $ R_{0} \leq 1$.	
\end{theorem}

\begin{proof}
Let
\begin{equation*}
w(t) = s(t)- s_0 - \int_{\frac{\Lambda}{\mu}}^{s(t)}
\lim\limits_{i \rightarrow 0^+} \dfrac{f(s_0,i(t))}{f( \sigma,i(t))}d\sigma + i(t)
+ k \displaystyle \int_0^h g(t) \displaystyle \int_{0}^{\tau} i(t-u)dud\tau
+ \dfrac{\delta}{\mu+\delta} r(t),
\end{equation*}
where
\begin{eqnarray*}
k = \dfrac{(\mu+\delta)(\mu+c+\gamma)-\delta\gamma}{\mu+\delta}.
\end{eqnarray*}
We have
\begin{equation*}
\begin{split}
\dfrac{dw(t)}{dt}
&= \left(1-\lim\limits_{i \rightarrow 0^+}
\dfrac{f(s_0,i(t))}{f( s(t),i(t))} \right) \left(\Lambda- \mu s(t)
- \beta \displaystyle \int_0^{h}g(\tau)f(s(t),i(t-\tau))d\tau \right) \\
&\quad + \beta  \displaystyle \int_0^{h}g(\tau)f(s(t),i(t-\tau))d\tau 
-(\mu+c+\gamma)i(t) +\delta r(t)\\
&\quad + \dfrac{\delta}{\mu+\delta} \left(\gamma i(t)- (\mu + \delta) r(t)\right)
+ k  \displaystyle \int_0^h g(t) \left(i(t)-i(t-\tau)\right)d\tau.
\end{split}
\end{equation*}
It follows that
\begin{equation*}
\begin{split}
\dfrac{dw(t)}{dt}
&= \mu\left(1-\lim\limits_{i \rightarrow 0^+}
\dfrac{f(s_0,i(t))}{f( s(t),i(t))} \right) \left(s_0-s(t) \right)\\
&\quad + \beta \lim\limits_{i \rightarrow 0^+}\dfrac{f(s_0,i(t))}{f( s(t),i(t))}
\displaystyle \int_0^{h}g(\tau)f(s(t),i(t-\tau)d\tau -(\mu+c+\gamma)i(t) \\
&\quad + \delta r(t) + \dfrac{\delta}{\mu+\delta} \left(\gamma i(t)
- \left(\mu + \delta\right) r(t) \right)
+ k  \displaystyle \int_0^h g(t) \left(i(t)-i(t-\tau)\right)d\tau\\
&= \mu \left(1-\lim\limits_{i \rightarrow 0^+}\dfrac{f(s_0,i(t))}{f( s(t),i(t))}\right)(s_0-s(t))\\
&\quad +  \displaystyle \int_0^{h}g(\tau) \left( \beta \dfrac{f(s(t),i(t-\tau))}{ki(t-\tau)}
\lim\limits_{i \rightarrow 0^+}\dfrac{f(s_0,i(t))}{f( s(t),i(t))}-1  \right)k i(t-\tau)d\tau.
\end{split}
\end{equation*}
Therefore,
\begin{eqnarray*}
\beta \dfrac{f(s(t),i(t-\tau))}{ki(t-\tau)} \lim\limits_{i \rightarrow 0^+}
\dfrac{f(s_0,i(t))}{f( s(t),i(t))} \leq \dfrac{\beta}{k}
\dfrac{\partial f(s_0,i(t))}{\partial i}= R_{0}.
\end{eqnarray*}
Then,
\begin{eqnarray*}
\dfrac{dw(t)}{dt}= \mu \left(1-\lim\limits_{i \rightarrow 0^+}
\dfrac{f(s_0,i(t))}{f( s(t),i(t))} \right) \left(s_0-s(t) \right)
+ \displaystyle \int_0^{h}g(\tau) \left(R_{0} -1\right)k i(t-\tau)d\tau.
\end{eqnarray*}
We conclude that $\dfrac{dw(t)}{dt}  \leq 0$. Hence, $w$
is a Lyapunov functional for the  system \eqref{Model}. Namely,
$w{'} \leq 0$ for all $ (s,i,r)\in \mathring{\Omega}$,
where $\mathring{\Omega}$ denotes the interior of $\Omega$.
Thus, $w{'} = 0$ if and only if $(s, i, r)=(s_0, 0, 0)$.
This shows that the largest invariant subset where $w{'} = 0$
is the singleton $\{E_0\}$. By La Salle's invariance principle,
$E_0$ is globally asymptotically stable. This completes the proof.
\end{proof}


\subsection{Global stability of the endemic equilibrium}

Now we look for the global asymptotic stability of the endemic equilibrium
$E^*$ of system \eqref{Model}. To this end, we construct a suitable
Lyapunov functional. We have the following result.

\begin{theorem}
\label{thm:Stab:EE}
Assume that the hypotheses $(H_1)$ and $(H_2)$ hold.
If $R_{0} > 1$, then the unique endemic equilibrium $E^*$
given by Lemma~\ref{Lemma:EE} is globally asymptotically stable.
\end{theorem}

\begin{proof}
Let $G(x)=x-1-\ln(x)$.
We have $G(x)\geq 0 $ if $x > 0$ and $G(x)=0 $ if $ x=1$.
Let us consider the following Lyapunov functional:
$V(t)=V_1(t)+V_2(t)+V_3(t)$, where
\begin{eqnarray*}
V_1(t) = s(t)-s^* - \int_{s^*} ^{s(t)} \dfrac{f(s^*,i^*)}{f( \sigma ,i^*)}
d \sigma + i(t)-i^* -i^* \ln\left(\dfrac{i(t)}{i^*}\right),
\end{eqnarray*}
\begin{equation*}
V_2(t) =  \beta f(s^*,i^*) \int_0^h g(\tau)\int_{0} ^{\tau}
G\left( \dfrac{ i(t-u)}{i^*} \right)du d\tau,
\quad
V_3(t) = \dfrac{\delta}{\mu+\delta}\left(r(t)-r^* -r^*
\ln\left(\dfrac{r(t)}{r^*}\right)\right).
\end{equation*}
Then,
\begin{multline*}
\dfrac{dV_1(t)}{dt}
= \left(1-\dfrac{f(s^*,i^*)}{f(s(t),i^*)} \right)
\left(\Lambda- \mu s(t)- \beta \displaystyle \int_0^{h}
g(\tau)f(s(t),i(t- \tau))d\tau \right)\\
+  \left(1-\dfrac{i^*}{i(t)}\right)
\left(\beta \displaystyle \int_0^{h}g(\tau)f(s(t),i(t-\tau))d\tau
-(\mu+c+ \gamma)i(t)+ \delta r(t) \right).
\end{multline*}
We also have
\begin{eqnarray*}
\dfrac{dV_2(t)}{dt}= \beta f(s^*,i^*)\int_0^{h}g(\tau)
\left(G\left( \dfrac{i(t)}{i^*}\right)
-G\left( \dfrac{i(t-\tau)}{i^*}\right)\right) d\tau
\end{eqnarray*}
and
\begin{eqnarray*}
\dfrac{dV_3(t)}{dt} = \dfrac{\delta}{\mu+\delta}
\left(1- \dfrac{r^*}{r(t)} \right)
\left(\gamma i(t)- (\mu + \delta) r(t) \right).
\end{eqnarray*}
One can see that
\begin{eqnarray*}
G \left( \dfrac{i(t)}{i^*} \right)
-G \left( \dfrac{i(t-\tau)}{i^*} \right)
= \dfrac{i(t)}{i^*}- \dfrac{i(t- \tau)}{i^*}
+ \ln \left( \dfrac{i(t- \tau)}{i(t)} \right).
 \end{eqnarray*}
Thus,
\begin{eqnarray*}
\left\{
\begin{array}{ll}
\Lambda= \mu s^* + \beta f(s^*,i^*),\\
\beta f(s^*,i^*)= ( \mu+c+ \gamma)i^*
- \dfrac{\delta\gamma}{\mu+\delta}i^*,\\
\gamma i^*=(\mu +\delta)r^*.
\end{array}
\right.
\end{eqnarray*}
It follows that
\begin{equation*}
\begin{split}
\dfrac{dV(t)}{dt}
&= \left(1-\dfrac{f(s^*,i^*)}{f(s(t),i^*)}\right)
\left(\mu s^* + \beta f(s^*,i^*)- \mu s(t) \right)\\
&\quad - \beta \left(1-\dfrac{f(s^*,i^*)}{f(s(t),i^*)}\right)
\displaystyle \int_0^{h} g(\tau)f(s(t),i(t- \tau))d\tau  \\
&\quad + \left(1-\dfrac{i^*}{i(t)} \right)\left(\beta
\displaystyle \int_0^{h}g(\tau)f(s(t),i(t-\tau))d\tau
-( \mu+c+ \gamma)i(t)\right)\\
&\quad +\left(1-\dfrac{i^*}{i(t)}\right) \delta r(t)\\
&\quad + \beta f(s^*,i^*) \displaystyle \int_0^{h}g(\tau)
\left(\dfrac{i}{i^*}- \dfrac{i(t- \tau)}{i^*}
+ \ln\left( \dfrac{i(t- \tau)}{i(t)}\right) \right) d\tau\\
&\quad + \dfrac{\delta}{\mu+\delta} \left(1- \dfrac{r^*}{r(t)} \right)
\left(\gamma i(t)- (\mu + \delta) r(t) \right) \\
&= \mu \left(1-\dfrac{f(s^*,i^*)}{f(s(t),i^*)}\right)
\left(s^* - s(t) \right)\\
&\quad + \beta f(s^*,i^*) \displaystyle \int_0^{h}g(\tau)
\left(\dfrac{f(s(t),i(t- \tau))}{f(s(t),i^*)}-\frac{i^*}{i(t)}
\dfrac{f(s(t),i(t- \tau))}{f(s^*,i^*)}\right)d\tau \\
&\quad +\beta f(s^*,i^*) \displaystyle \int_0^{h}g(\tau)
\left(2-\frac{f(s^*,i^*)}{f(s(t),i^*)}-\frac{i(t)}{i^*}\right)d\tau \\
&\quad + \beta f(s^*,i^*) \displaystyle \int_0^{h}g(\tau)
\left(\dfrac{i}{i^*}- \dfrac{i(t- \tau)}{i^*}
+ \ln\left( \dfrac{i(t- \tau)}{i(t)}\right) \right) d\tau \\
&\quad + \left( 2 \dfrac{\delta\gamma}{\mu+\delta}i^*
- \delta \frac{i^*}{i(t)} r(t)
-\dfrac{\delta\gamma}{\mu+\delta}\frac{i(t)}{r(t)} r^* \right).
\end{split}
\end{equation*}
Using
\begin{equation*}
\ln\left( \dfrac{i(t- \tau)}{i(t)}\right)
= \ln \left(\dfrac{f(s^*,i^*)}{f(s(t),i^*)}\right)
+ \ln \left(\frac{i^*}{i(t)} \dfrac{f(s(t),i(t- \tau))}{f(s^*,i^*)} \right)
+ \ln \left(\frac{i(t- \tau)}{i^*} \dfrac{f(s(t),i^*)}{f(s(t),i(t- \tau))}\right),
\end{equation*}
then
\begin{equation*}
\begin{split}
\dfrac{dV(t)}{dt}
&= \mu \left(1-\dfrac{f(s^*,i^*)}{f(s(t),i^*)} \right)
\left(s^* - s(t) \right) \\[2ex]
&\quad +  \beta f(s^*,i^*) \displaystyle \int_0^{h}g(\tau)
\left(1- \dfrac{f(s^*,i^*)}{f(s(t),i^*)}
+\ln \left(\dfrac{f(s^*,i^*)}{f(s(t),i^*)}\right) \right) d\tau \\[2ex]
&\quad +  \beta f(s^*,i^*) \displaystyle \int_0^{h}g(\tau)
\left(1- \frac{i^*}{i(t)} \dfrac{f(s(t),i(t- \tau))}{f(s^*,i^*)}\right) d\tau\\[2ex]
&\quad +  \beta f(s^*,i^*) \displaystyle \int_0^{h}g(\tau)\ln
\left( \frac{i^*}{i(t)} \dfrac{f(s(t),i(t- \tau))}{f(s^*,i^*)} \right) d\tau \\[2ex]
&\quad +  \beta f(s^*,i^*) \displaystyle \int_0^{h}g(\tau)
\left(1- \frac{i(t- \tau)}{i^*} \dfrac{f(s(t),i^*)}{f(s(t),i(t- \tau))}\right)d\tau 
\end{split}
\end{equation*}
\begin{equation*}
\begin{split}
&\quad +\beta f(s^*,i^*) \displaystyle \int_0^{h}g(\tau)
\ln \left( \frac{i(t- \tau)}{i^*} \dfrac{f(s(t),i^*)}{f(s(t),i(t- \tau))} \right) d\tau \\[2ex]
&\quad + \beta f(s^*,i^*) \displaystyle \int_0^{h}g(\tau)
\left( \frac{i(t- \tau)}{i^*} \dfrac{f(s(t),i^*)}{f(s(t),i(t- \tau))} - 1 \right) d\tau\\[2ex]
&\quad + \beta f(s^*,i^*) \displaystyle \int_0^{h}g(\tau)
\left(- \frac{i(t- \tau)}{i^*} + \dfrac{f(s(t),i(t- \tau))}{f(s(t),i^*)} \right) d\tau\\[2ex]
&\quad - \dfrac{\delta\gamma}{\mu+\delta}\frac{i(t)}{r(t)} r^*
\left(\dfrac{i^*r(t)}{i(t)r^*}-1 \right)^2.
\end{split}
\end{equation*}
The hypotheses $(H_1)$ and $(H_2)$ ensure that
\begin{eqnarray*}
\frac{i(t- \tau)}{i^*} \dfrac{f(s(t),i^*)}{f(s(t),i(t- \tau))}
- 1- \frac{i(t- \tau)}{i^*} + \dfrac{f(s(t),i(t- \tau))}{f(s(t),i^*)}
\leq  0.
\end{eqnarray*}
From the hypothesis $(H_1)$, we have
\begin{eqnarray*}
\mu \left(1-\dfrac{f(s^*,i^*)}{f(s(t),i^*)} \right) \left(s^* - s(t) \right) \leq 0.
\end{eqnarray*}
Hence, $\dfrac{dV(t)}{dt} \leq 0$. We conclude that the endemic equilibrium
of system \eqref{Model} is globally asymptotically stable
in $\mathring{\Omega}$, provided $R_{0} > 1$.
\end{proof}


\section{Numerical simulations}
\label{sec:05}

In this section, we carry out some numerical simulations
to illustrate the obtained theoretical results. Consider
the following delayed SIRI epidemic model with distributed
time delay and relapse:
\begin{eqnarray}
\label{example1}
\left\{
\begin{array}{lll}
\dfrac{ds(t)}{dt}
= \Lambda- \mu s(t)- \displaystyle \beta \int_0^{h}
\dfrac{ e^{- \tau}}{1- e ^{-h}}s(t)i(t-\tau) d \tau,\\[2ex]
\dfrac{di(t)}{dt}= \displaystyle \beta \int_0^{h}
\dfrac{ e^{- \tau}}{1- e^{-h}}s(t)i(t- \tau) d \tau
-( \mu+c+ \gamma)i(t) +\delta r(t), \\[2ex]
\dfrac{dr(t)}{dt}=\gamma i(t)- (\mu +\delta) r(t)
\end{array}
\right.
\end{eqnarray}
with initial conditions
\begin{equation}
\label{ex1:ic01}
\Phi_1(\theta)=\sin(0.5 \theta) + 150,
\ \  \  \Phi_2(\theta) = \sin(10 \theta) + 20,
\  \  \  \Phi_3(\theta) = 0,
\quad -h\leq \theta \leq 0,
\end{equation}
or
\begin{equation}
\label{ex1:ic02}
\Phi_1(\theta)=\cos(5 \theta) + 200,
\ \    \  \Phi_2(\theta) = 10 \sin( \theta) + 30,
\  \  \  \Phi_3(\theta) = 70,
\quad -h\leq \theta \leq 0.
\end{equation}
From Figure~\ref{figure1}, we see that if $R_{0} \leq 1$,
then, biologically speaking, the disease dies 
out from the population: the solution $(s(t),i(t),r(t))$ 
of \eqref{example1} converges to the disease free equilibrium 
$E_0=\left(\frac{\Lambda}{\mu},0,0\right)$. Hence, $E_0$ is globally 
asymptotically stable. 
\begin{figure}[htbp]
\centering
\subfigure[$s(t)$]{\includegraphics[scale=.33]{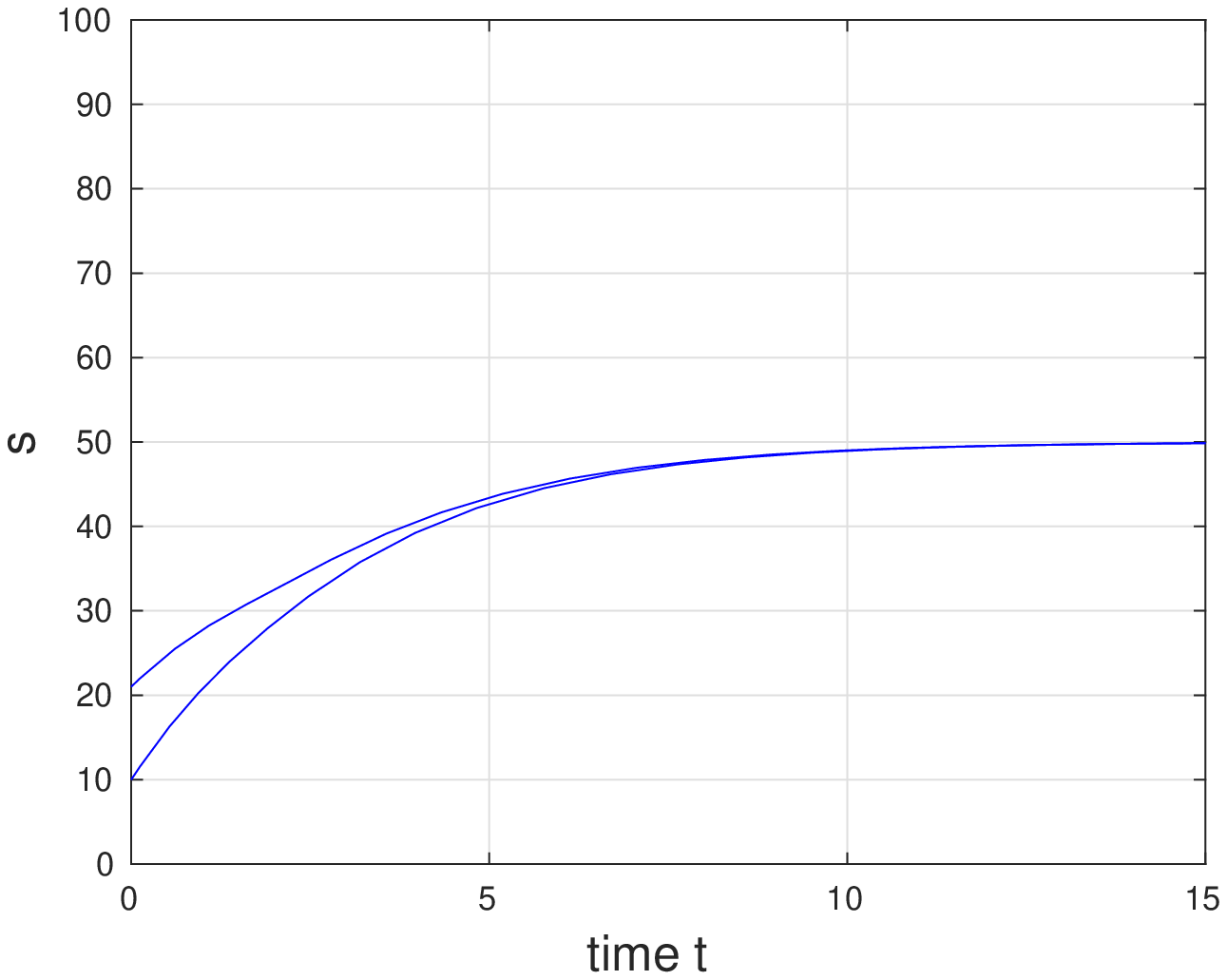}}
\subfigure[$i(t)$]{\includegraphics[scale=.33]{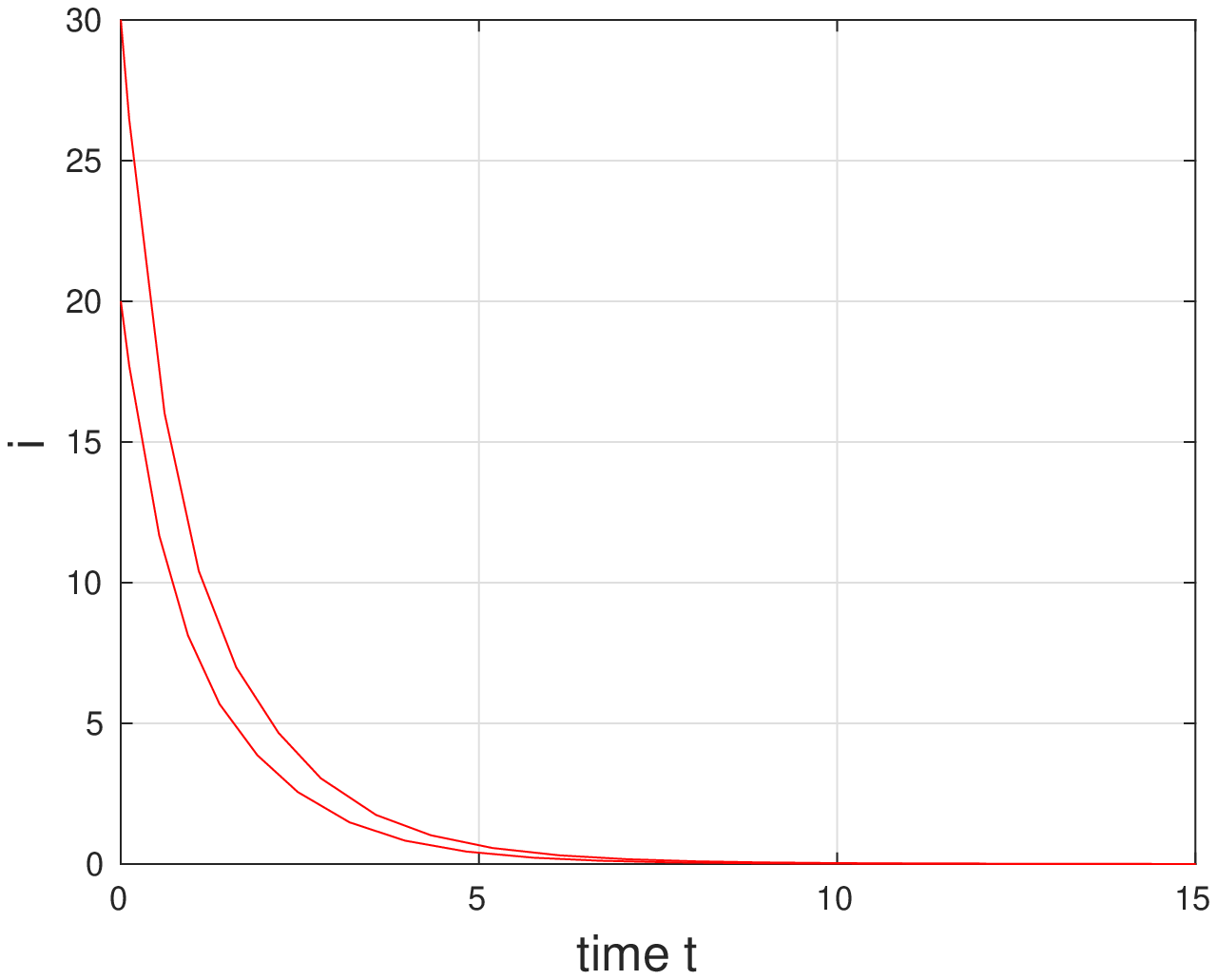}}
\subfigure[$r(t)$]{\includegraphics[scale=.33]{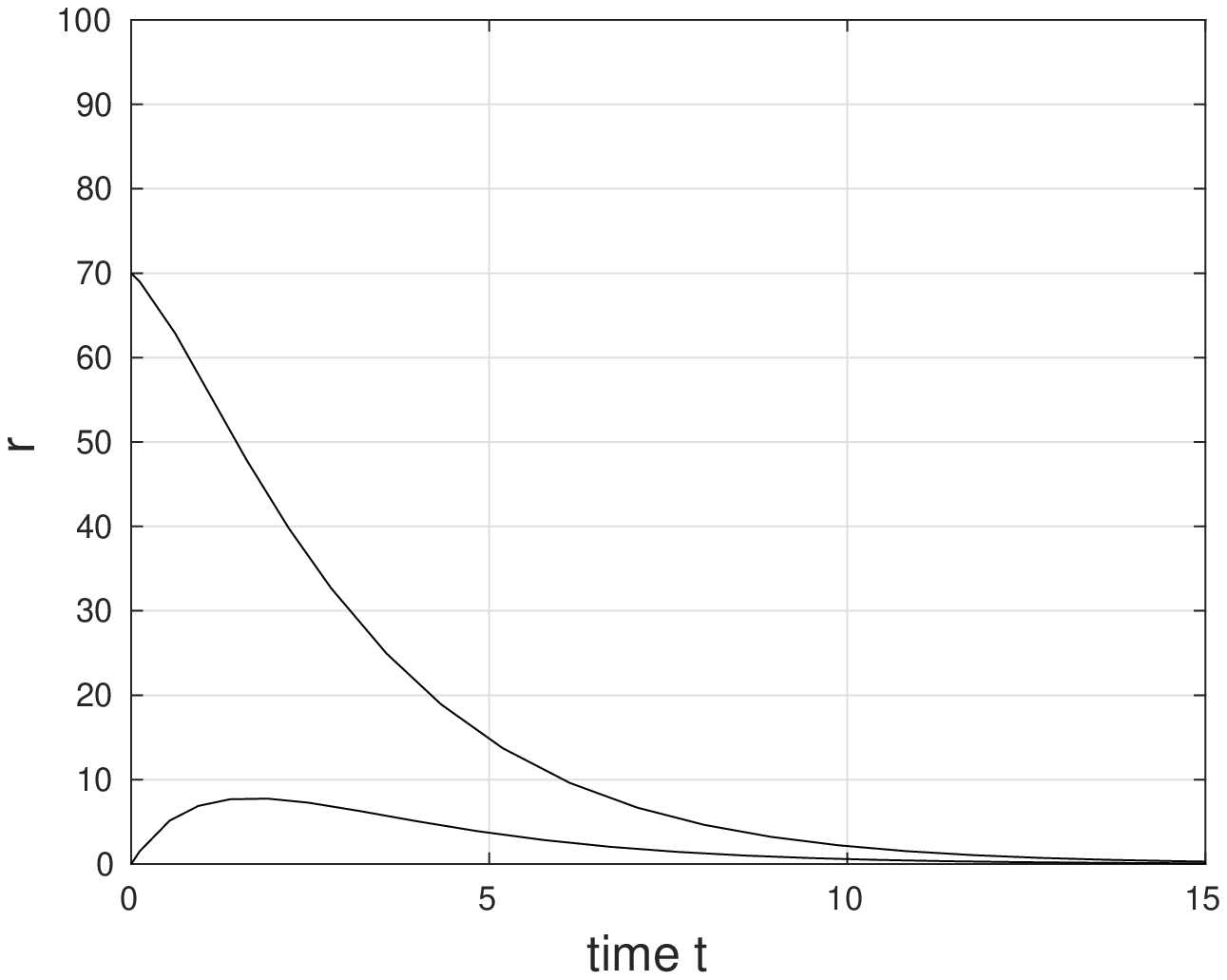}}
\caption{Trajectories of system \eqref{example1} 
with time delay $h = 2$, initial conditions \eqref{ex1:ic01}
and \eqref{ex1:ic02}, and parameters $\Lambda = 20$, $\mu = 0.4$, 
$\beta = 0.02$, $c = 0.1$, $\delta = 0.006$, 
and $ \gamma = 0.7$, for which $R_{0} = 0.8406 < 1$,
converging to the disease free equilibrium $E_0$.}
\label{figure1}
\end{figure}

Figure~\ref{figure2} illustrates the case when $R_{0} > 1$. 
In such situation, the solution $ (s(t),i(t),r(t))$ of \eqref{example1}
converges to the endemic equilibrium $E^*$. Thus, the unique endemic
equilibrium is globally asymptotically stable, which means, biologically,
that the disease persists but is controlled.	
\begin{figure}[htbp]
\centering
\subfigure[$s(t)$]{\includegraphics[scale=.348]{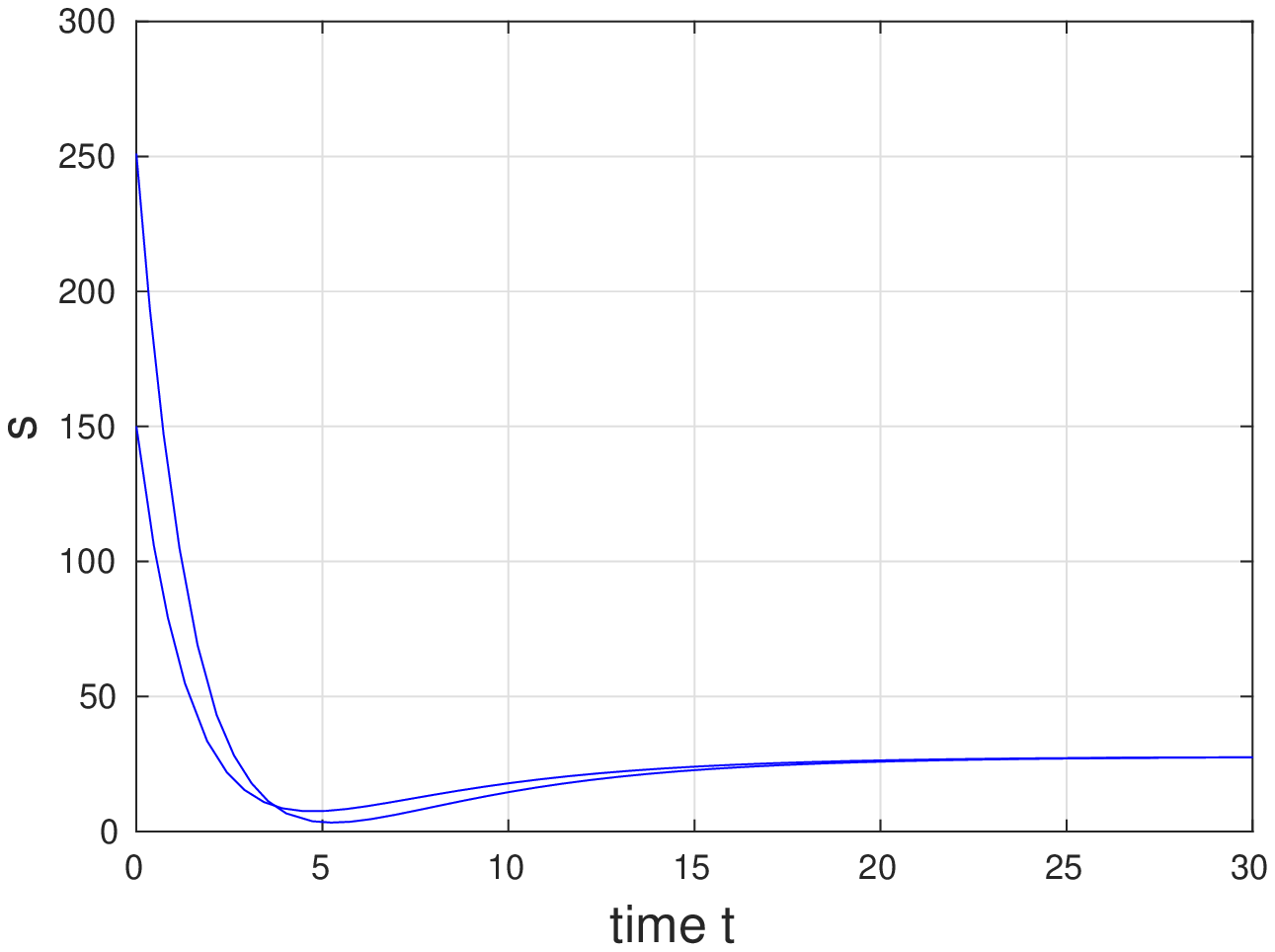}}
\subfigure[$i(t)$]{\includegraphics[scale=.328]{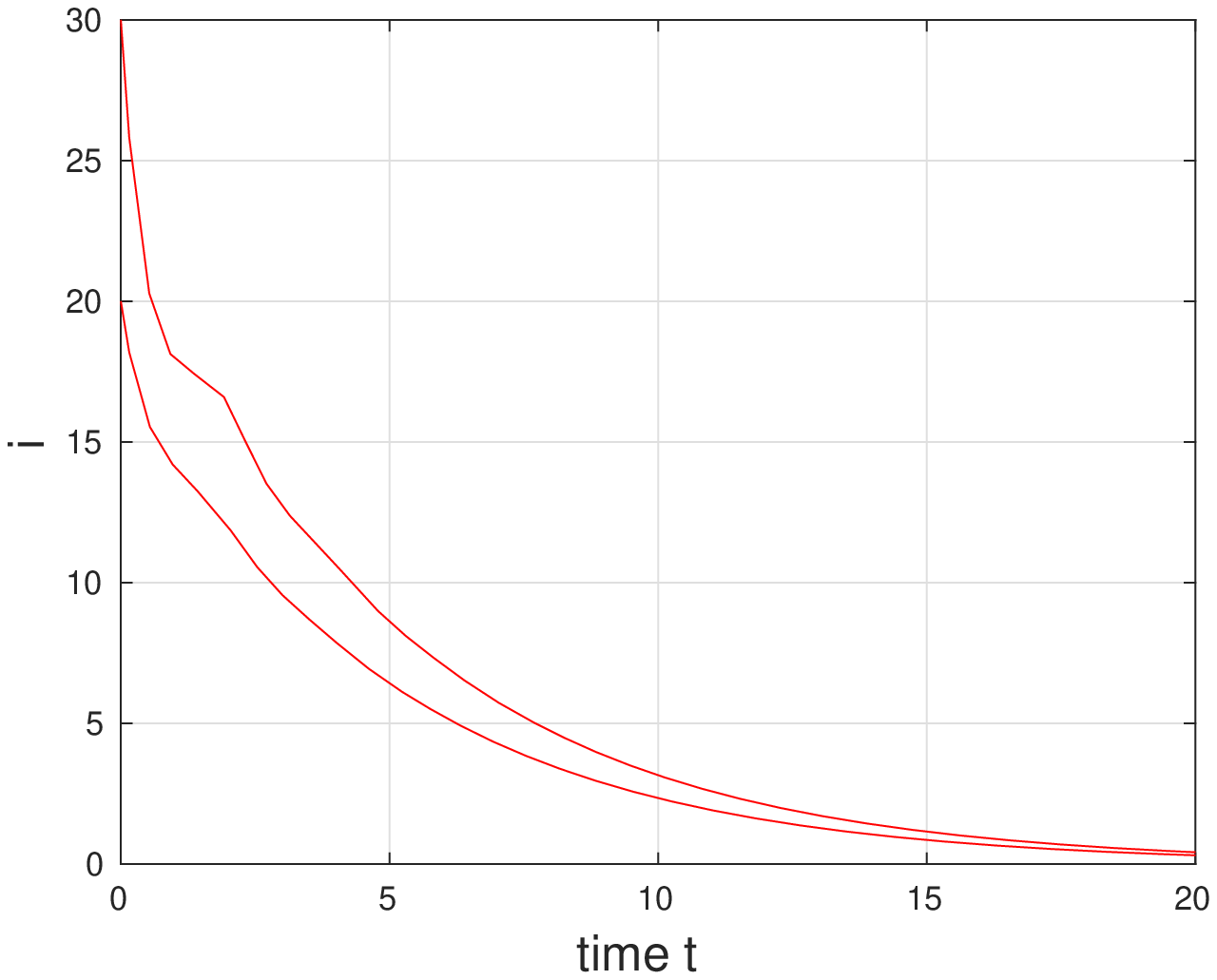}}
\subfigure[$r(t)$]{\includegraphics[scale=.328]{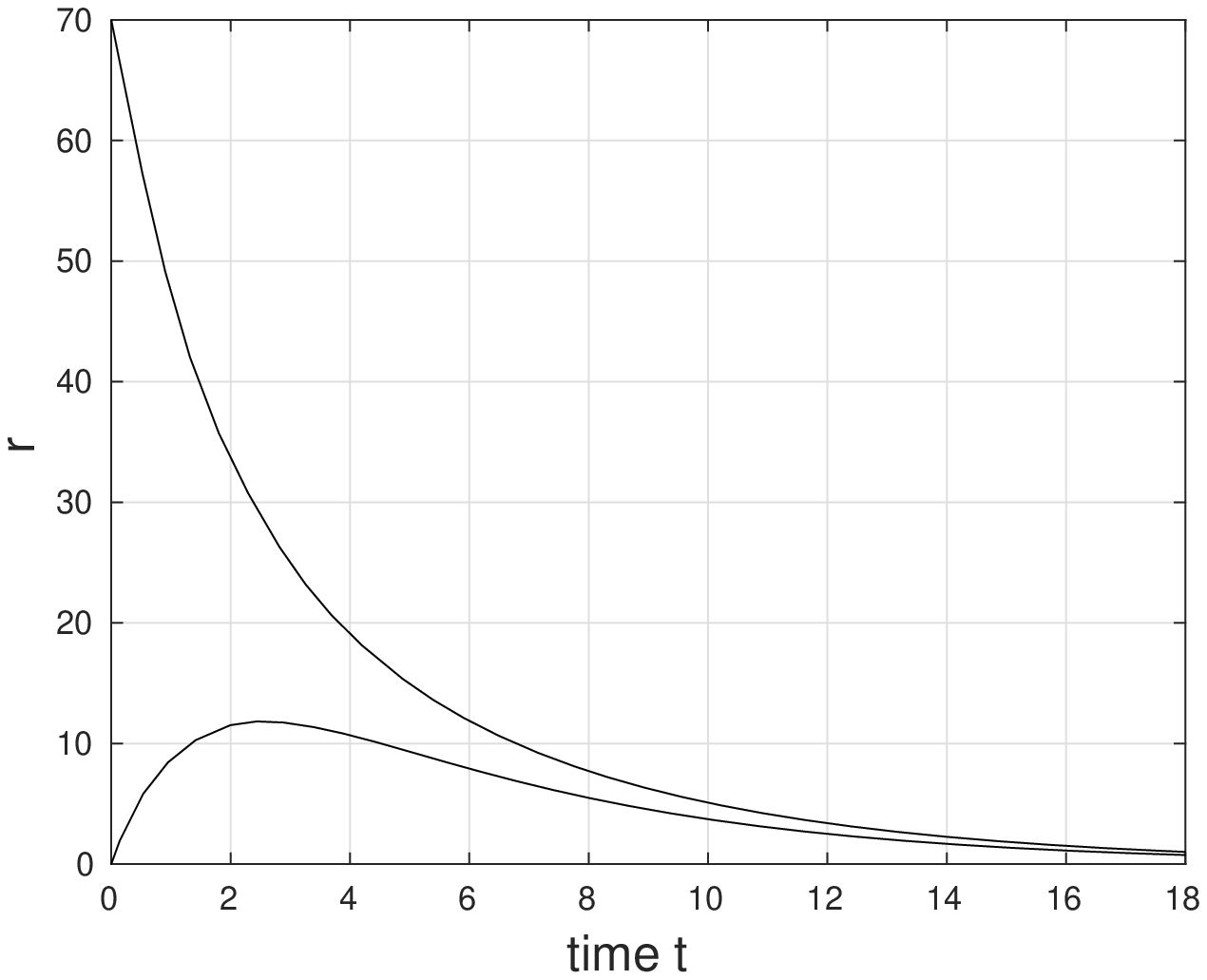}}
\caption{Trajectories of system \eqref{example1} with time delay
$h = 2$, initial conditions \eqref{ex1:ic01}
and \eqref{ex1:ic02}, and parameters $\Lambda = 18$, $\mu = 0.65$, 
$\beta = 0.2$, $c = 0.77$, $\delta = 0.02$, and $ \gamma = 0.75$, 
for which $R_{0} = 2.2923 > 1$, converging to the endemic 
equilibrium $E^*$.}
\label{figure2}
\end{figure}


\section{Concluding remarks}
\label{sec:06}

We investigated a SIRI epidemic model \eqref{Model}
with distributed delay and relapse. The basic reproduction 
number $R_{0}$ was computed, which determines the existence 
of an equilibrium for the model. Precisely, when $R_{0} \leq 1$, 
then model \eqref{Model} has one unique disease-free equilibrium $E_0$, 
while for $R_{0} > 1$ it has a disease free equilibrium $E_0$ 
and a unique endemic equilibrium $ E^*$. We proved, with the help 
of the direct Lyapunov method, that all steady states of \eqref{Model} 
are globally asymptotically stable: the disease free equilibrium is 
globally asymptotically stable for $R_{0}\leq 1$; when $ R_{0} > 1$, 
then we established that there is a unique endemic equilibrium which 
is globally asymptotically stable. As future work, we plan to study 
a related model with distributed relapse. This is under current 
investigation and will be addressed elsewhere.


\section*{Appendix}

\begin{theorem}[See \cite{budincevic2010comparison}]
\label{thm:05:App}
Let $f, g: A\rightarrow\mathbb{R}$, with $A\subset\mathbb{R}^{2}$ an open set, 
be continuous Lipschitz functions with respect to the second argument such that
$f(t, x) \leq g(t, x)$ for all $(t, x) \in A$. Moreover, let $(t_0, y_0)$ and 
$(t_0, v_0)$ be two admissible initial states with $y_0 \leq v_0$. 
Then, if $y, v : I \rightarrow \mathbb{R}$ are, respectively, 
the solutions of the Cauchy problems	
\begin{eqnarray*}
\left\{
\begin{array}{ll}
\dfrac{dy}{dt} = f(t, y(t)),\\[0.2cm]
y(t_0) = y_0,
\end{array}
\right.
\end{eqnarray*}
and
\begin{eqnarray*}
\left\{
\begin{array}{ll}
\dfrac{dv}{dt} = g(t, v(t)),\\[0.2cm]
v(t_0) = v_0,
\end{array}
\right.
\end{eqnarray*}
where $I$ is the common interval of existence, then
$y(t) \leq v(t)$ for all $t  \in I$ and $t\geq t_0$.
\end{theorem}


\section*{Acknowledgements}

This work was partially supported by FCT through the
R\&D unit CIDMA, reference UID/MAT/04106/2019,
and by project PTDC/EEI-AUT/2933/2014 (TOCCATA).
The authors would like to thank two Reviewers for their 
critical remarks and precious suggestions, which helped 
them to improve the quality and clarity of the manuscript. 



\end{document}